\documentclass{amsart}\usepackage[]{graphicx}\usepackage[]{color}
\makeatletter
\def\maxwidth{ %
  \ifdim\Gin@nat@width>\linewidth
    \linewidth
  \else
    \Gin@nat@width
  \fi
}
\makeatother

\definecolor{fgcolor}{rgb}{0.345, 0.345, 0.345}

\usepackage{framed}
\makeatletter
 {\par\unskip\endMakeFramed%
 \at@end@of@kframe}
\makeatother

\definecolor{shadecolor}{rgb}{.97, .97, .97}
\definecolor{messagecolor}{rgb}{0, 0, 0}
\definecolor{warningcolor}{rgb}{1, 0, 1}
\definecolor{errorcolor}{rgb}{1, 0, 0}
\newenvironment{knitrout}{}{} 

\usepackage{alltt}

\usepackage{ifxetex}
\ifxetex
 \usepackage{fontspec}
\else
 \usepackage[utf8]{inputenc}
\fi

\usepackage[backend=biber,style=authoryear]{biblatex}
\usepackage{booktabs}
\usepackage{bbold}
\usepackage{amsthm}

\newtheorem{theorem}{Theorem}
\newtheorem{corollary}[theorem]{Corollary}
\newtheorem{definition}[theorem]{Definition}
\newtheorem{method}{Method}
\IfFileExists{upquote.sty}{\usepackage{upquote}}{}
\begin{document}

\title{Certified Mapper: Repeated testing for acyclicity and obstructions to the nerve lemma}
\author{Mikael Vejdemo-Johansson and Alisa Leshchenko}

\begin{abstract}
The Mapper algorithm does not include a check for whether the cover produced conforms to the requirements of the nerve lemma. To perform a check for obstructions to the nerve lemma, statistical considerations of multiple testing quickly arise. 

In this paper, we propose several statistical approaches to finding obstructions: through a persistent nerve lemma, through simulation testing, and using a parametric refinement of simulation tests.

We suggest Certified Mapper -- a method built from these approaches to generate certificates of non-obstruction, or identify specific obstructions to the nerve lemma -- and we give recommendations for which statistical approaches are most appropriate for the task.
\end{abstract}

\maketitle

\section{Introduction}
\label{sec:introduction}

The Mapper model has found widespread use since its initial creation \parencite{singh_topological_2007,lum_extracting_2013}.

The Mapper method starts with a topological space $X$ (often, but not always a point cloud in some $\mathbb{R}^d$) paired with a map $f:X\to C$ to some coordinate space $C$ with a cover $C=\bigcup_{i\in\mathcal I} C_i$. 
The cover $\mathcal C=\{C_i\}$ pulls back to a cover $f^{-1}\mathcal C$ of $X$, which can be refined by replacing each cover element $C_i$ by its connected components $\pi_0C_i$ to form a refined cover $\pi_0f^{-1}\mathcal C$.
This refined cover is the \emph{Mapper cover} of $X$, and the \emph{Mapper complex} is the nerve complex $\mathcal N\pi_0f^{-1}\mathcal C$.

On data, the topological space is replaced by a finite metric space $(X,d)$, and the connected components functor replaced by a clustering scheme.

\subsection{Main results}
\label{sec:main-results}

In this paper, we 
\begin{enumerate}
\item Define a Mapper certificate: an indication that a Mapper cover admits the Nerve lemma
\item Define obstructions to a Mapper certificate
\item Prove that a separation condition suffices to use the Persistent Nerve Lemma (Theorem \ref{thm:persistent-nerve-lemma}) to produce a certificate (or an obstruction)
\item Suggest several methods for statistical testing where the separation condition is not fulfilled. Among the suggested methods, we provide explanations for the methods that fail, and power analysis and validation for the methods that succeed.
\end{enumerate}

A Certified Mapper analysis -- Mapper with a certificate of non-obstruction -- brings additional surety of fidelity of shape to the Mapper analysis, through the applicability of nerve lemmata to the Mapper cover.

\section{Nerve Lemmata}
\label{sec:nerve-lemmata}

The production of a cover and use of a nerve complex in Mapper evokes the \emph{Nerve lemma} and its variants:

\begin{theorem}[Nerve lemma]
Let $X$ be a topological space and $\mathcal U=\{U_i\}$ a cover of $X$.
If $\mathcal U$ is a \emph{good cover} -- each non-empty intersection $U_{i_1}\cap\dots\cap U_{i_k}$ is contractible -- then $X$ is homotopy equivalent to the nerve complex $\mathcal N\mathcal U$.
\end{theorem}

Contractible means homotopy equivalent to the one-point topological space.
The homotopy conditions and statement can be relaxed to a homological nerve lemma

\begin{theorem}[Homological nerve lemma]
Let $X$ be a topological space and $\mathcal U=\{U_i\}$ a cover of $X$.
If $\mathcal U$ is a \emph{good cover} -- each non-empty intersection $U_{i_1}\cap\dots\cap U_{i_k}$ is acyclic -- then $H_*X$ is isomorphic to $H_*\mathcal N\mathcal U$.
\end{theorem}

Acyclic means the homology is isomorphic to the homology of the one-point topological space.
\parencite{govc2016approximate} proved a \emph{persistent homology nerve lemma}. 
We will be using this in Section~\ref{sec:obstructions}.

\begin{theorem}[Persistent homology nerve lemma]\label{thm:persistent-nerve-lemma}
Let $X$ be a filtered simplicial complex and $\mathcal U=\{U_i\}$ a cover of $X$.
If $\mathcal U$ is an $\epsilon$-\emph{good cover} -- each non-empty intersection $U_{i_1}\cap\dots\cap U_{i_k}$ is $\epsilon$-interleaved with the empty complex -- then persistent $H_nX$ is $2(n+1)\epsilon$-interleaved with persistent $H_n\mathcal N\mathcal U$.
\end{theorem}

\section{Obstructions}
\label{sec:obstructions}

If the \emph{good cover} condition of the nerve lemma fails, then the topology can change arbitrarily much: hidden topological features can both create and remove topological structure when passing from $X$ to $\mathcal N\mathcal U$.

For each of these nerve lemmata, locating any one cover element intersection where the corresponding good cover condition is not fulfilled produces an obstruction to the equivalency produced by that nerve lemma. 
A lack of obstruction could be taken as an indication that a topological description of the nerve complex is an appropriate description of the original space.

\begin{definition}
An \textbf{obstruction in dimension $\mathbf{d}$} is a significantly persistent $k$-homology class in a $d-k+1$-fold intersection of cover elements.

A \textbf{certificate in dimension $\mathbf{d}$} for a Mapper cover is a documented absence of obstructions in dimension $d$.
\end{definition}

The persistent homology nerve lemma produces our first method for detecting and quantifying obstructions.
Key to using this is the construction of a filtered simplicial complex on the point cloud $X$ using its cover elements $U_i$.

We define two point clouds to be $\epsilon$-\emph{separated} if 
\[
\min_{x\in U_i\setminus U_j, y\in U_j\setminus U_i} d(x,y) > \epsilon
\]

\begin{theorem}\label{thm:epsilon-acyclic}
If a point cloud $X$ is covered by sub-point clouds $\mathcal U=\{U_i\}$ such that each pair of cover elements $U_i, U_j$ are $\epsilon$-separated, then the \v Cech (Vietoris-Rips) complex with radius at most $\epsilon$ of $X$ is covered by the \v Cech (Vietoris-Rips) complexes with radius at most $\epsilon$ on the cover elements.
\end{theorem}
\begin{proof}
If a simplex $[x_0,\dots,x_d]$ is in the \v Cech (Vietoris-Rips) complex of $X$, but not in any one complex of a cover element, then there are vertices $x, y$ and cover elements $U_i, U_j$ such that $x\in U_i\setminus U_j$ and $y\in U_j\setminus U_i$. 
Hence, if the cover is $\epsilon$-separated, such an obstructing simplex can only occur at a radius greater than $\epsilon$.
\end{proof}

\begin{corollary}\label{cor:epsilon-acyclic}
If a point cloud $X$ is covered by sub-point clouds $\mathcal U=\{U_i\}$ such that:

\begin{enumerate}
\item each pair of cover elements $U_i, U_j$ are $\epsilon$-separated
\item each cover element intersection is $\epsilon$-acyclic
\item the oldest death time is at most $\epsilon$
\end{enumerate}

then persistent $n$-homology of \v Cech (Vietoris-Rips) complex of $X$ is $2(n+1)\epsilon$-interleaved with persistent $n$-homology of $\mathcal N\check{C}_*\mathcal U$ ($\mathcal NVR_*\mathcal U$).

Hence, if the point cloud is embedded in $\mathbb R^d$, then the complexes are $2(d+1)\epsilon$-interleaved across all homological dimensions.
\end{corollary}

\section{Statistical acyclicity}
\label{sec:statistical-acyclicity}

If the conditions of Theorem \ref{thm:epsilon-acyclic} or Corollary \ref{cor:epsilon-acyclic} are not fullfilled, more work needs to be done to quantify obstructions. In the following sections we will propose and evaluate a number of possible approaches to statistical testing and quantification of obstructions. Ultimately, we will find one strongly recommended method, and a few properties that disqualify otherwise promising ideas.

An inviting statistical approach may decide on a numeric invariant for measuring acyclicity -- persistence length of the most persistent feature of reduced homology, or some symmetric function in the sense of \parencite{adcock2016ring} or a tropical symmetric function \parencite{kalivsnik2014symmetric} -- and measure whether the invariant produced by the cover $\mathcal U$ of the point cloud $X$ is significantly larger than the invariants produced by simulating complexes using some model of persistent homologically trivial barcodes.

This produces a first naïve method for testing acyclicity in a point cloud:
\begin{method}[Generic simulation test of acyclicity]\label{mth:generic}
Given a point cloud $X$, an invariant $\gamma:\{\text{Point clouds}\}\to\mathbb R$, and a null model $\mathcal M$ of random point clouds, we may reject the null hypothesis of acyclicity in favor of non-acyclicity by:
\begin{enumerate}
\item Draw $M_1,\dots,M_{N-1}$ from $\mathcal M$
\item Compute all $\gamma(M_j)$ and $\gamma(X)$.
\item Sort all these $N$ values, and let $r$ be the rank of $\gamma(X)$.
\end{enumerate}
We may then reject the null hypothesis at a level of $p=\frac{N-r+1}{N}$.
\end{method}

\parencite{bobrowski2017maximally} observe that uniformly sampled points in a cube have small largest persistence lengths, and conjecture that the persistence ratios ($d/b$ for a persistent homology class that appears at time $b$ and vanishes again at time $d$) are normally distributed.
This suggests that one useful null model for trivial random point clouds would be a uniform distribution on the bounding box of the point cloud we compare against.

Tradition in persistent homology suggests $\max d-b$, and \parencite{bobrowski2017maximally} suggests $\max d/b$ as useful invariants for measuring acyclicity. 
Both of these invariants have a tendency to vary in scale between different homological dimensions, but for the ratio invariant there is a conjecture of it following a normal distribution.

The ratio invariant is not defined for homological dimension 0, a case not studied by \parencite{bobrowski2017maximally}.
This means that it would not be able to pick up 0-dimensional homology classes in testing.
During preliminary studies, the ratio invariant provided no noticable difference between simulations with signal and simulations without.

When testing for a good cover, however, there is one point cloud for each simplex in the nerve complex -- and since even one single rejection forms an obstruction to the nerve lemma, we need to control for the family-wise error rate (FWER: probability of a single false rejection) rather than the false discovery rate (FDR: expected proportion of false rejections).
To adequately handle these error rates we will need more intricate methods than Method \ref{mth:generic}: we need to either apply an appropriate control method for FWER, or find a new concept of an appropriate statistic so that the simulation test no longer suffers from repeated testing issues.
We will explore candidates for more adequately controlling for FWER in Methods \ref{mth:normal-T} and \ref{mth:quantile-T}, and candidates for changing the simulation statistic in Method \ref{mth:global}. 

\subsection{Null models}
\label{sec:null-model}

All of the simulation based methods rely on being able to draw random point clouds from a null model $\mathcal M$, that models what a contractible space \emph{should} look like.

From the work in \parencite{bobrowski2017maximally} we know that uniform distributions tend to have very short persistence intervals, while the work in \parencite{adler2014crackle} suggests that multivariate normal random data would tend to produce quite large persistence intervals by the weight of the tails of the distribution.

Based on this we would suggest that an equal number of points sampled uniformly from a shape derived from the point cloud we are trying to match would be an appropriate model.

Two shapes immediately suggest themselves for use: we could use a convex hull or an axis-aligned bounding box as a container implied by the data.
Both of these, taken as is, will produce biased results since in both cases data points are on the boundary of the region -- an unlikely result in the case of sampling uniformly at random, since the boundary has measure 0.

We do not know how to produce an unbiased enlargement of the convex hull.
For the bounding box, however, we can estimate the bounds $a, b$ in each dimension separately assuming that the points come from a null model of a uniform distribution on some interval $[a,b]$. 
For coordinates $x_1,\dots,X_N$, such an unbiased estimator is given by
\[
\hat a = \frac{N\cdot\min(x_i)-\max(x_i)}{N-1}
\qquad
\hat b = \frac{N\cdot\max(x_i)-\min(x_i)}{N-1}
\]

\subsection{Adjusting thresholds: Bonferroni, Holm, Hochberg}
\label{sec:bonferroni}

A widely used family of methods for controlling FWER is by adjusting the thresholds of rejection: to have an overall level of $\alpha$ for a hypothesis test, each separate test out of a family of $K$ tests is rejected at a level of $k\alpha/K$ for $k$ some constant depending on the aggregated p-values from the different tests.

Bonferroni correction, known to be overly conservative, uses the sub-additivity of probability measures to suggest a constant $k=1$.
Improved versions include the Holm step-down and the Hochberg step-up processes, both of which derive the $k$ multipliers used from a ranking of the p-values.

For these methods, the number $N-1$ of simulations will be dependent on the size of the nerve complex: with $K$ cells and a desired level of $\alpha$ the simulation load is on the order of $K/\alpha$.
For large covers, the increase in computational load quickly becomes prohibitive.

With an acyclicity test that includes the sizes of the statistics used rather than only their ranks, these correction methods become more accessible: if the non-trivial topology produces a much larger invariant value than the null model, the p-values involved in the correction procedures can shrink below $1/N$.

\subsubsection{Normal approximation of maximal ratios}
\label{sec:normal-approximation-ratio}

If we assume the conjecture in \parencite{bobrowski2017maximally}, the persistence ratios are normally distributed.
We can estimate the mean and variance of these persistence ratios from simulations, and then compare the values directly to the corresponding normal distribution.

Though there is no reason to expect normality for the maximum difference invariant, we could (and do) evaluate the same test built on that invariant as well.

\begin{method}[Normal test of ratio acyclicity]\label{mth:normal-T}
Given a point cloud $X$, and a null model $\mathcal M$ of random point clouds, we may reject the null hypothesis of acyclicity in favor of non-acyclicity by:
\begin{enumerate}
\item Draw $M_1,\dots,M_{N-1}$ from $\mathcal M$
\item Compute all $\gamma(M_j)$ and $\gamma(X)$. Compute 
\[
\mu=\frac{\sum\gamma(M_j)}{N-1} \qquad 
S^2=\frac{\sum(\gamma(M_j)-\mu)^2}{N-2} \qquad
Z=\frac{\gamma(X)-\mu}{S}\sim T(1)\\
\]
\end{enumerate}
We may then reject the null hypothesis at a level of $p=1-\text{CDF}_{\mathcal N(0,1)}(Z)$.
\end{method}

\subsubsection{Normal approximation of quantiles}
\label{sec:normal-approximation-quantile}

\parencite{keener2011theoretical} gives an asymptotic normal distribution for quantiles of arbitrary distributions (with differentiable distribution functions).

\begin{theorem}\label{thm:normal-quantile}
Let $X_1,\dots,$ be iid with common cumulative distribution function $F$, let $\gamma\in(0,1)$ and let $\hat\theta_n$ be the $\lfloor\gamma n\rfloor$ order statistic for $X_1,\dots,X_n$ (or a weighted average of the $\lfloor\gamma n\rfloor$ and the $\lceil\gamma n\rceil$).

If $F(\theta)=\gamma$ and $F'(\theta)$ exists, is finite and positive, then

\[
\sqrt{n}(\hat\theta_n-\theta)\Rightarrow
\mathcal N\left(
0, 
\frac
{\gamma(1-\gamma)}
{F'(\theta)^2}
\right)
\]
\end{theorem}

We do not know whether diagram invariants will follow a differentiable distribution function -- but if it did, we could use several batches of simulations of the null model to get a statistic with a known variance:
\begin{enumerate}
\item Given a point cloud $X$ and simulations $M_2,\dots,M_N$, we can calculate $\gamma(X)$ and all $\gamma(M_j)$.
\item By ranking all these values, we can find the quantile $q$ of $X$ in this simulated batch. Then
\[
\sqrt{N}(\gamma(X)-F^{-1}(q))\sim
\mathcal N\left(0, \frac{q(1-q)}{F'(F^{-1}(q))^2}\right)
\]
\item If we calculate another $N$ simulations $K_1,\dots,K_N$, and estimate the $q$ quantile $K_{(q)}$, we also know
\[
\sqrt{N}(\gamma(K_{(q)})-F^{-1}(q))\sim
\mathcal N\left(0, \frac{q(1-q)}{F'(F^{-1}(q))^2}\right)
\]
\item We can subtract one expression from the other to yield
\[
\sqrt{N}(\gamma(X)-\gamma(K_{(q)}))\sim
\mathcal N\left(0, 2\frac{q(1-q)}{F'(F^{-1}(q))^2}\right)
\]
\item If we calculate yet another $N$ simulations $L_1,\dots,L_N$ and estimate the $q$ quantile $L_{(q)}$ we also have
\[
\sqrt{N}(\gamma(L_{(q)})-\gamma(K_{(q)}))\sim
\mathcal N\left(0, 2\frac{q(1-q)}{F'(F^{-1}(q))^2}\right)
\]
\item Write $V=2\frac{q(1-q)}{F'(F^{-1}(q))^2}$. Then
\[
\frac{\sqrt{N}(\gamma(X)-\gamma(K_{(q)}))}{V}
{\Large/}
\frac{\sqrt{N}(\gamma(L_{(q)})-\gamma(K_{(q)}))}{V} =
\frac{\gamma(X)-\gamma(K_{(q)})}{\gamma(L_{(q)})-\gamma(K_{(q)})}
\]
is a quotient of two standard normal variables; this distributes as a $T(1)$ random variable.
\end{enumerate}

Knowing the distribution of the ratio we can use the $T(1)$ distribution to build a hypothesis test:

\begin{method}[Ratio T-test of quantile acyclicity]\label{mth:quantile-T}
Given a point cloud $X$, an invariant $\gamma:\{\text{Point clouds}\}\to\mathbb R$, and a null model $\mathcal M$ of random point clouds, we may reject the null hypothesis of acyclicity in favor of non-acyclicity by:
\begin{enumerate}
\item Draw $M_1,\dots,M_{N-1}$ from $\mathcal M$
\item Compute all $\gamma(M_j)$ and $\gamma(X)$. Compute the rank $r$ of $\gamma(X)$ among all these values. Write $x=\gamma(X)$.
\item Draw $M'_1,\dots,M'_N$ and $M''_1,\dots,M''_N$. Write $y$ for the $r$th value among the $M'_*$ and $z$ for the $r$th value among the $M''_*$. 
\item Calculate the test statistic
\[
T = \frac{(x-z)}{V}{\Large/}\frac{(y-z)}{V} = 
\frac{x-z}{y-z}\sim T(1)
\]
\end{enumerate}
We may then reject the null hypothesis at a level of $p=1-\text{CDF}_{T(1)}(T)$.
\end{method}

\subsection{Empirical distributions and normalized maximal persistences}
\label{sec:ecdf-normalized}

Instead of driving down the p-values to comply with a classical control mechanism, we may instead change perspective on the simulation testing.
This approach was developed in conversations with Sayan Mukherjee.

Many invariants of persistence bars differ with the overall scale of the point cloud, so the invariants are not immediately comparable.
If they were, however, then the existence of an obstruction in the cover would be witnessed by the largest value of an invariant.
Therefore a joint test can be built on first making the invariants comparable, and then doing a simulation test where in each simulation step the largest invariant value is extracted.

To make persistence diagram invariants comparable, we suggest two potential approaches for standardization:
\begin{enumerate}
\item If the invariant $\gamma$ are (sufficiently close to) normally distributed, we can studentize our invariant values separately within each local point cloud and its simulated nulls.
\item We can use a non-parametric standardization method, such as histogram equalization within each local point cloud and its simulated nulls.
\end{enumerate}

Based on this we propose the following approach
\begin{method}[Standardized global test of acyclicity]\label{mth:global}
Given a family of point clouds $X_1,\dots,X_K$, an invariant $\gamma:\{\text{Point clouds}\}\to\mathbb R$, and a null model $\mathcal M$ of random point clouds, we may reject the null hypothesis of acyclicity in favor of non-acyclicity by:
\begin{enumerate}
\item Draw $M^1_1,\dots,M^{N-1}_{K}$ from $\mathcal M$.
\item Compute all $\tilde y_i^j = \gamma(M_i^j)$ and $\tilde x_i\gamma(X_i)$.
\item For each $i\in[1,K]$, use $\tilde y_i^j$ to create a standardization method, (ie to calculate mean and standard deviation for the studentization, or to calculate the empirical CDF for histogram equalization) and standardize all $\tilde y_i^j$ to $y_i^j$ and standardize $\tilde x_i$ to $x_i$. 
\item For each $j\in[1,N-1]$ calculate $y_i=\max_j y_i^j$. 
Calculate $x=\max x_i$.
\item Compute the rank $r$ of $x$ among $x$ together with all the $y_i$.
\end{enumerate}
We may then reject the null hypothesis at a level of $p=(N-r+1)/N$.
\end{method}

\section{Experiments}
\label{sec:experiments}

To validate our suggested methods and compare their performances we perform simulation tests on null model data input to verify the level of each correction method, and with a single noisy circle input together with null model data input for a power analysis of each method. 

We use the null model of uniformly distributed points in a plane rectangle, and for computational expediency we restrict our testing to two ambient dimensions.

Our simulations test for all combinations of:

\begin{itemize}
\item $N\in\{100,500\}$ (number of point clouds for each test)
\item $K\in\{5,10,50\}$ (number of simultaneous tests to control)
\end{itemize}

For each box, we draw uniformly at random

\begin{itemize}
\item Box side lengths in $\{0.1,1,10\}$
\item Point counts for a box in $\{10,50,100,500\}$
\item For the power test: in one of the boxes, points on a circle with added multivariate isotropic Gaussian noise with variance from $\{0.1, 0.25\}$ fitted in a square box with side lengths $1\times 1$.
\end{itemize}

The $\alpha$-complex construction is topologically equivalent to \v Cech complexes \parencite{bauer2014morse}, and for speed in our simulations we choose to use the $\alpha$-complex persistent homology calculation in the R package TDA \parencite{fasy2014tda}. 
With simulations in place we perform bootstrap evaluations of level and power of all combinations of:

\begin{itemize}
\item Methods \ref{mth:normal-T}, \ref{mth:quantile-T}, \ref{mth:global} for controlling the FWER.
\item FWER correction with Hochberg's method, standardization with Z-score and histogram equalization.
\end{itemize}

We will use the invariant $\gamma(X) = \max d-b$ of maximum bar length.

We illustrate the process of computing a certification on a real world data set in Section \ref{sec:real-world-data}.

\section{Results}
\label{sec:results}

We will divide our simulation study results into three components: first we will examine the suggestion of a $T(1)$ distribution in Method \ref{mth:quantile-T}.
Next, we validate the FWER control procedures by estimating the probability of false discovery on null model data.
Finally, we will analyze the power of the proposed methods by attempting to detect a single noisy circle in a family of null model data samples.

For the experiments, we precomputed 160000 point cloud invariants.
Since we are working with point clouds in the plane, we computed in homological dimensions 0 and 1, and for each combination of box shapes and point counts as well as for each noise level and point count combination, we generated 5000 point clouds.
All our subsequent results are based on drawing from these precomputed invariants at random, matching box sizes and point counts when producing simulations to match a particular point cloud.

\subsection{Validation}
\label{sec:validation}

\begin{figure}

\begin{knitrout}
\definecolor{shadecolor}{rgb}{0.969, 0.969, 0.969}\color{fgcolor}
\includegraphics[width=\maxwidth]{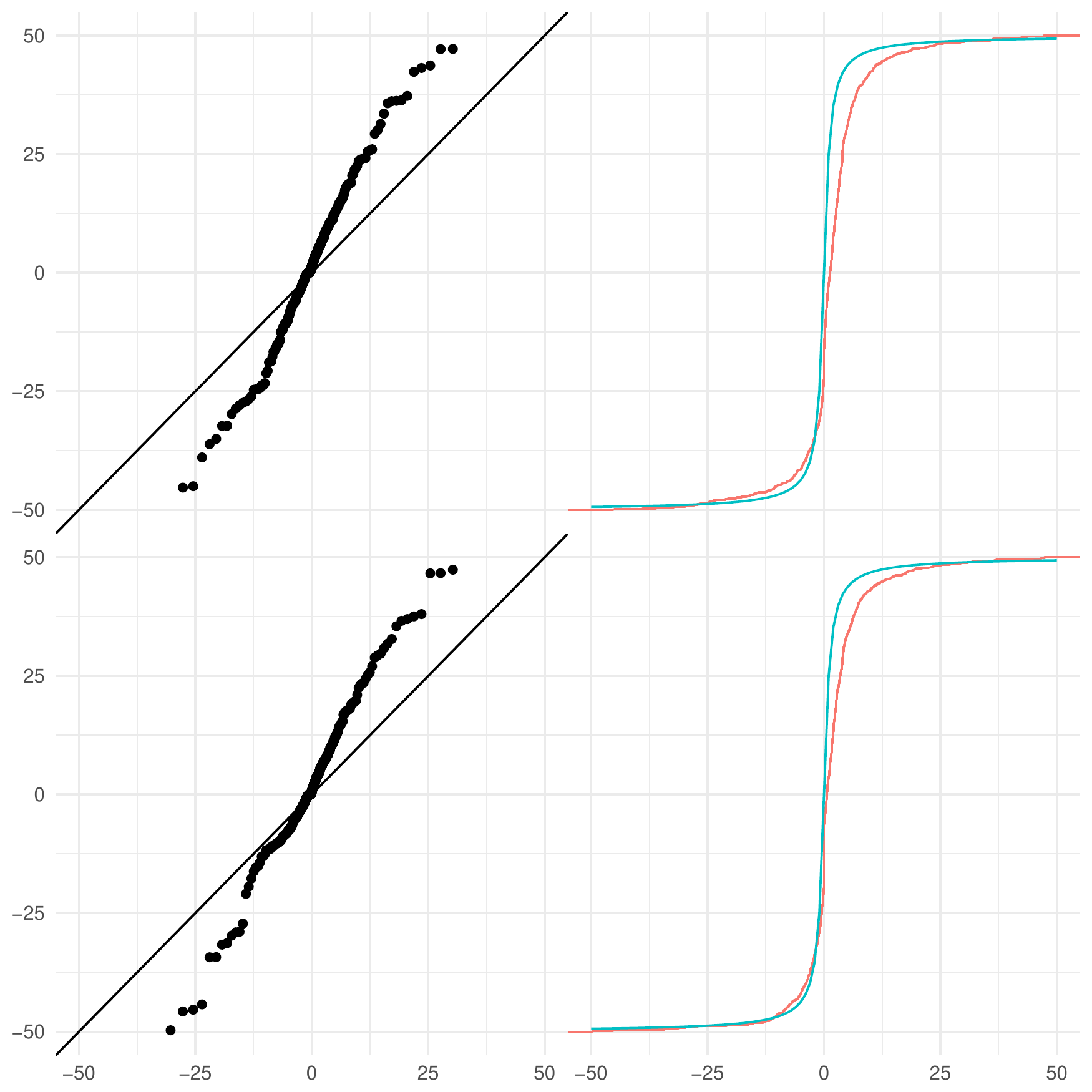} 

\end{knitrout}
\caption{QQ-plots and ECDF-plots for evaluating goodness of fit of $2T$, with $T$ the ratio from Method \ref{mth:quantile-T} against the $T(1)$ distribution. The very heavy tails of the T-distribution produce very large outliers in the tails of the distributions: we have pruned the plots for readability here. \newline
The fit to the $T(1)$ distribution is still not very good, \emph{after} adding the factor of 2 (found empirically). We cannot right now say why this factor was needed.\newline
Top row is the result from using the difference invariant in homological dimension 0 and the bottom row is the difference invariant in homological dimension 1.}
\label{fig:normal-quantile-distr}
\end{figure}

The first claim to validate is the applicability of Theorem \ref{thm:normal-quantile} to the numeric invariant data we would be getting from persistence barcodes.
We discover empirically that for the test statistic $T$ from Method \ref{mth:quantile-T}, we get a better fit to the $T(1)$ distribution using the quantity $2T$, based on 1000 simulated values.
The $T(1)$ distribution has \emph{very} heavy tails -- as a result, the fit remains bad in the tails due to how commonly too large results appear in simulations. We display plots here for the central part of the points, together with the line $y=x$ in the QQ-plots to give a reference for how a perfect fit would be expressed.

Next, we evaluate the empirical level of our proposed methods.
From 100 simulations drawing from pre-computed barcode sizes, the null rejection rates for null model data for our methods are summarized in Table \ref{tbl:p-power}.
For each of the simulations, a random number, between 2 and 50 of point cloud invariants were drawn from the precomputed data.
To each point cloud invariant, another 99 point clouds with matching box sizes and point counts are drawn as a simulation test.
These 100 batches of 100 point clouds go through each of our proposed methods, and rejection rates at confidence levels of 0.1, 0.05 and 0.01 are calculated.

\subsection{Power estimation}
\label{sec:power}

\begin{table}
\begin{tabular}{lcccccc}
\toprule
& \multicolumn{3}{c}{Parametric} & \multicolumn{3}{c}{Global} \\ 
\cmidrule(lr){2-4}\cmidrule(lr){5-7}
$p<$ & Normal test & Log normal test & Quantile T-test & Z-score & log Z-score & Histogram Eq   \\
\\ & \multicolumn{6}{c}{Null model} \\ \cmidrule(lr){2-7} 
0.01 & 0.25 & 0.04 & 0 & 0.01 & 0.02 & 0\\
0.05 & 0.48 & 0.11 & 0.01 & 0.02 & 0.07 & 0\\
0.10 & 0.55 & 0.18 & 0.05 & 0.1 & 0.13 & 0.01\\

\\ & \multicolumn{6}{c}{$\sigma=0.1$} \\ \cmidrule(lr){2-7} 
0.01 & 0.97 & 0.79 & 0.02 & 0.86 & 0.77 & 0\\
0.05 & 0.98 & 0.81 & 0.05 & 0.95 & 0.79 & 0.04\\
0.10 & 0.99 & 0.82 & 0.15 & 0.97 & 0.82 & 0.06\\

\\ & \multicolumn{6}{c}{$\sigma=0.25$} \\ \cmidrule(lr){2-7} 
0.01 & 0.65 & 0.37 & 0 & 0.28 & 0.28 & 0\\
0.05 & 0.82 & 0.52 & 0.01 & 0.5 & 0.42 & 0.02\\
0.10 & 0.85 & 0.55 & 0.04 & 0.58 & 0.49 & 0.07\\

\bottomrule
\end{tabular}
\caption{Rejection rates for null model and noisy circle data using the difference and ratio invariants, and using the methods described above. 
The parametric methods are the methods that rely on an explicit distribution followed by a FWER control method: Method \ref{mth:normal-T} with normal and log-normal distribution assumptions, and Method \ref{mth:quantile-T}.
The global methods refer to Method \ref{mth:global} with either a normal Z-score, log-normal Z-score or histogram equalization method for standardization.\newline
FWER control was performed using Hochberg's method.}
\label{tbl:p-power}
\end{table}

For the power analysis we picked pre-calculated invariants from circles with a $1\times 1$ bounding box, with additive multivariate Gaussian noise with a standard deviation of $0.1$ and $0.25$ respectively.
For each of 100 simulations, one circle invariant was picked, and another random number (between 1 and 49) of null model point cloud invariants added.
This collection of point clouds go through the same process of generating 100-1 null model invariants for each, and run the collections through the described methods.
The result of 100 simulations each at the two noise levels is shown in Table \ref{tbl:p-power}.

\begin{figure}
\begin{knitrout}
\definecolor{shadecolor}{rgb}{0.969, 0.969, 0.969}\color{fgcolor}
\includegraphics[width=\maxwidth]{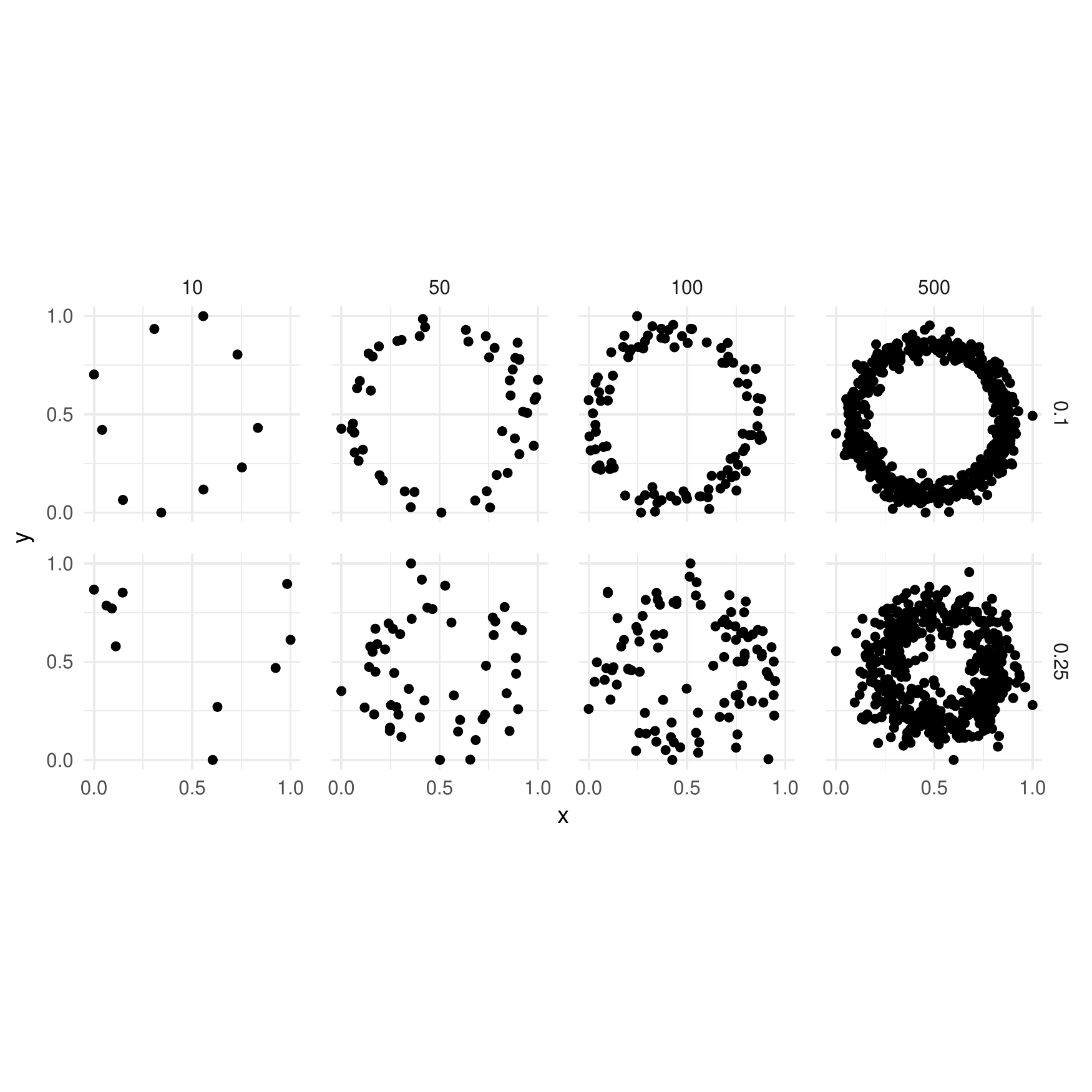} 

\end{knitrout}
\caption{Noisy circles as used by the power calculation. Top row, $\sigma=0.1$ and bottom row $\sigma=0.25$. The plots have, from left to right, 10, 50, 100 and 500 points.}
\label{fig:noisy-circle}
\end{figure}

Examples of the kind of circles we use for the power calculation can be seen in Figure \ref{fig:noisy-circle}

\subsection{Simulated data}
\label{sec:simulated-data}

To show an obstruction in action, we generated 250 random points on the cartesian product of a cross with a circle:
\[
X = (\Delta[0,1]\cup\Delta'[0,1])\times S^1
\qquad
\Delta(x) = (x,x)
\qquad
\Delta'(x) = (x,-x)
\]

The result is a 4-dimensional dataset in the shape of two pipes that intersect in the middle.

A Mapper analysis using the first coordinate as a filter function, with 10 divisions and a 50\% overlap was calculated using \texttt{TDAmapper}. The dataset and the resulting Mapper analysis can be seen in Figure \ref{fig:tubes}.

\begin{figure}
\begin{knitrout}
\definecolor{shadecolor}{rgb}{0.969, 0.969, 0.969}\color{fgcolor}
\includegraphics[width=\maxwidth]{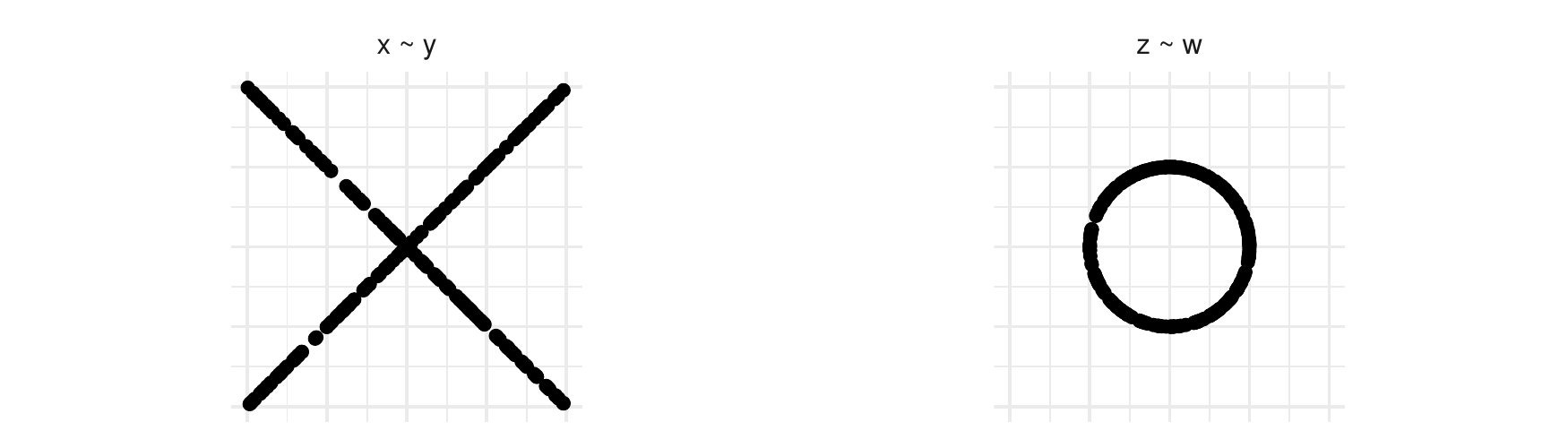} 

\end{knitrout}

\includegraphics[width=\textwidth]{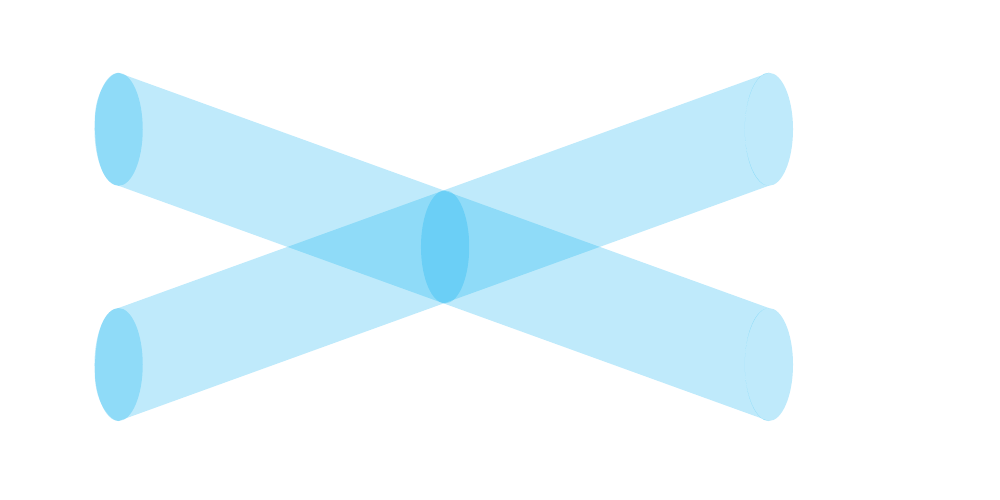}

\begin{knitrout}
\definecolor{shadecolor}{rgb}{0.969, 0.969, 0.969}\color{fgcolor}
\includegraphics[width=\maxwidth]{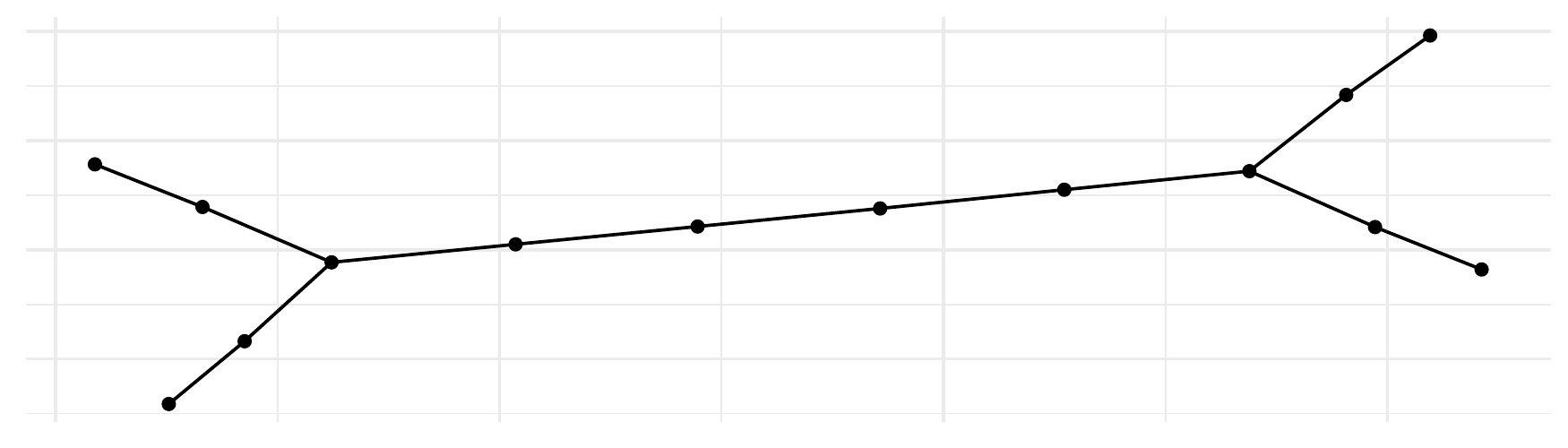} 

\end{knitrout}

\begin{knitrout}
\definecolor{shadecolor}{rgb}{0.969, 0.969, 0.969}\color{fgcolor}
\includegraphics[width=\maxwidth]{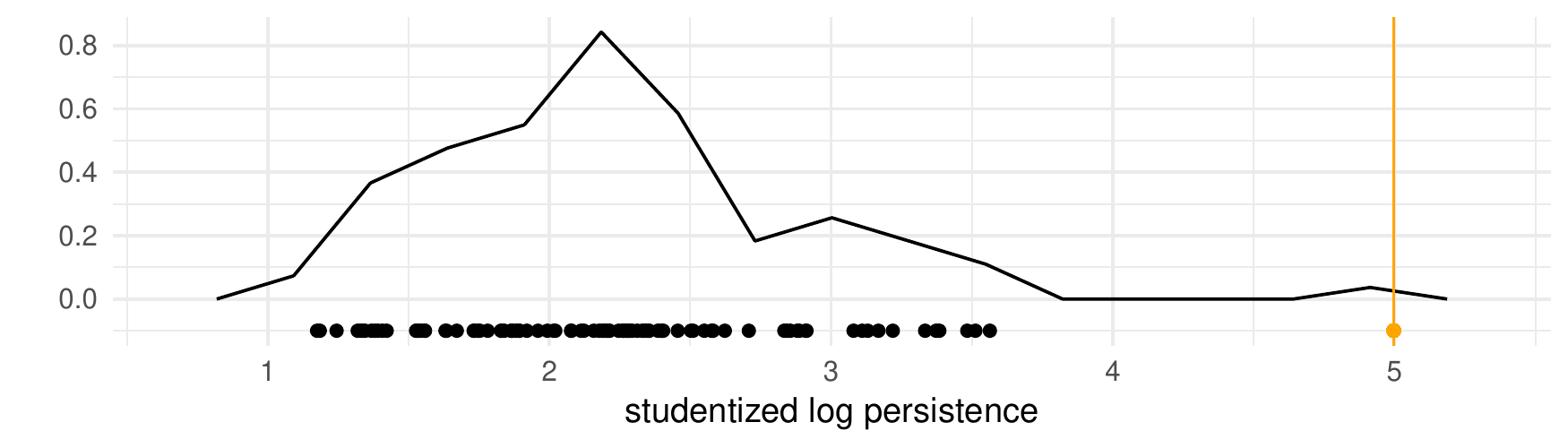} 

\end{knitrout}
\caption{Top left: the dataset in the $x-y$-plane. Top right: the dataset in the $z-w$-plane.\newline
Middle: the Mapper graph produced. The $\times$-like shape is clearly captured by the Mapper analysis, but the $z-w$ circle is absent.\newline
Bottom: frequency curve of the maximal studentized log persistence lengths for each of the $99$ simulations in addition to the dataset itself. Marked in orange and with a vertical line is the corresponding score for the dataset itself.}
\label{fig:tubes}
\end{figure}

We used Corollary \ref{cor:epsilon-acyclic} and Method \ref{mth:global} with the null model described in Section \ref{sec:null-model}.
To standardize we used Z-scores of log persistence lengths.

First, to use the Corollary, we would look for the maximum of lifespans and death times in the data.
This value comes out to $1.11$.
If the sections are $1.11$-separated this would show us that the Mapper graph and the Vietoris-Rips graph on the data were $4.43$-interleaved.
This amount of separation is unlikely, since the bounding box of the entire dataset comes out to $2\times 2\times 1\times 1$ and sliced into 10 slices along the first axis.

The Corollary conditions having failed, we turn to the probabilistic approach.
Using $99$ simulations we get the distribution seen at the bottom of Figure \ref{fig:tubes}.
From a visual inspection, the dataset is a clear outlier -- by ranking the maximal Z-scores over each of the simulations, the dataset comes in at rank 
$100$
for an upper-tailed $p$-value of 
$0.01$
(estimated using the $(N-r+1)/N$ estimate as given by \parencite{davison1997bootstrap})

With a significant result, we can find at least on obstruction by looking for a node or edge with a large Z-score associated to its persistent homology.
The largest Z-scores within the real data is in the $20$th of the simplices (in the ordering generated by our enumeration) which works out to the simplex [7]. The corresponding data points are graphed in Figure \ref{fig:obstruction}.

\begin{figure}
\begin{knitrout}
\definecolor{shadecolor}{rgb}{0.969, 0.969, 0.969}\color{fgcolor}
\includegraphics[width=\maxwidth]{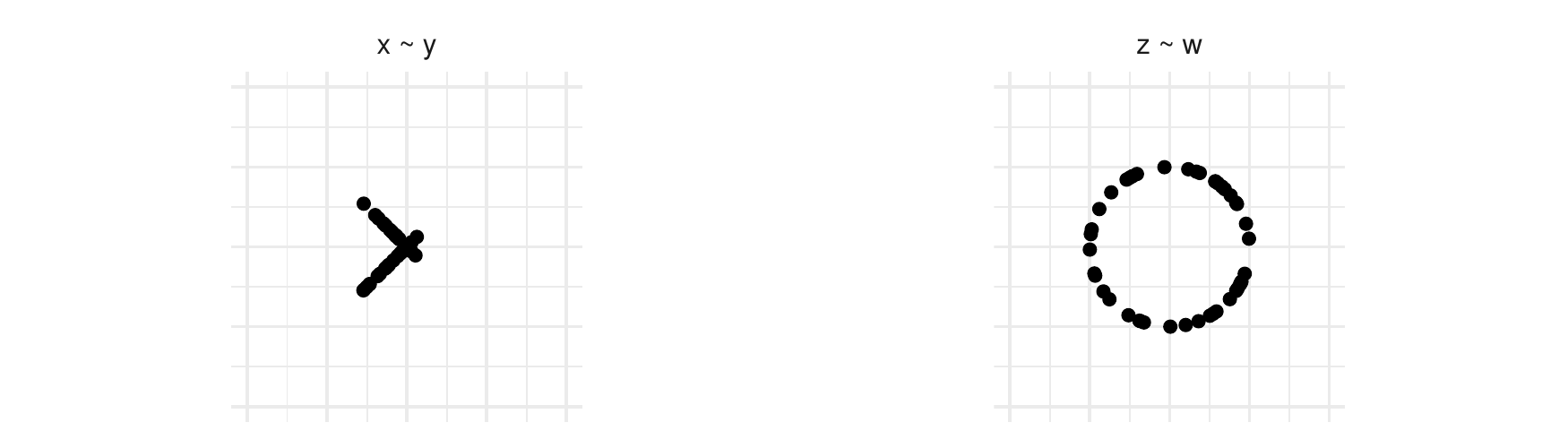} 

\end{knitrout}
\caption{The data subset witnessing the obstruction of highest significance found in the dataset.}
\label{fig:obstruction}
\end{figure}

\subsection{Real world data}
\label{sec:real-world-data}

We ran a Mapper analysis on Fisher's and Anderson's Iris dataset \parencite{fisher36,anderson35}, with a single filter function given by the Petal Length variable, with 10 divisions and a 50\% overlap was calculated using \texttt{TDAmapper}.
Next we ran the results through the certification process with the method suggested by Corollary \ref{cor:epsilon-acyclic} and Method \ref{mth:global} standardized using Z-scores of log persistence lengths.
The resulting Mapper graph can be seen in Figure \ref{fig:iris}.

For the interleaving distance from Corollary \ref{cor:epsilon-acyclic}, we calculate the maximum of lifespans and death times in the data.
This maximum comes out to $0.93$. 
Since we are using a single filter function, the Mapper complex is one-dimensional, so the multiplier for Corollary \ref{cor:epsilon-acyclic} is $4$ and the Corollary tells us that if the sections are $0.93$-separated, then the persistent homology of the resulting filtered graph is $3.73$-interleaved with the true persistent homology of the original dataset.

The Petal-Width variable has a total spread of 2.4, and our Mapper analysis uses 10 sections: a separation of $0.93$ seems highly unlikely.
Hence, the Corollary does not apply and we are forced to look towards probabilistic certification.

Using Method \ref{mth:global}, we calculated 99 simulations in addition to the true data.
We used the unbiased bounding box as a null-model
Values were standardized using the logarithm of the Z-score, as estimated on the simulated values.
For each simulation, the maximal log Z-score were selected across the Mapper graph. 
The distribution of these values can be seen in Figure \ref{fig:iris}.
As the graph indicates, there is no significant obstruction in the data, and by estimating an upper-tailed $p$-value as $(N-r+1)/N$ where $r$ is the rank of the log Z-score from the data set we get a $p$-value of 
$0.59$.

\begin{figure}
\begin{knitrout}
\definecolor{shadecolor}{rgb}{0.969, 0.969, 0.969}\color{fgcolor}
\includegraphics[width=\maxwidth]{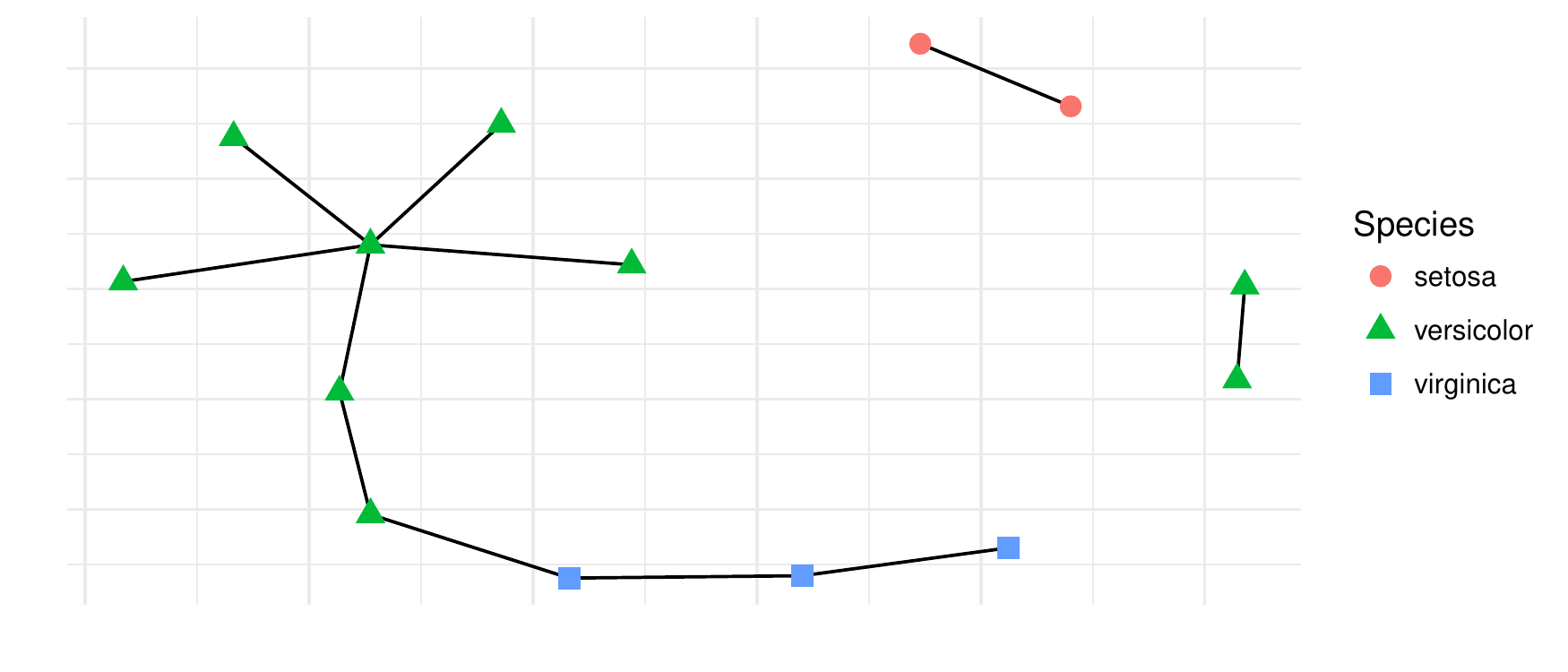} 

\end{knitrout}

\begin{knitrout}
\definecolor{shadecolor}{rgb}{0.969, 0.969, 0.969}\color{fgcolor}
\includegraphics[width=\maxwidth]{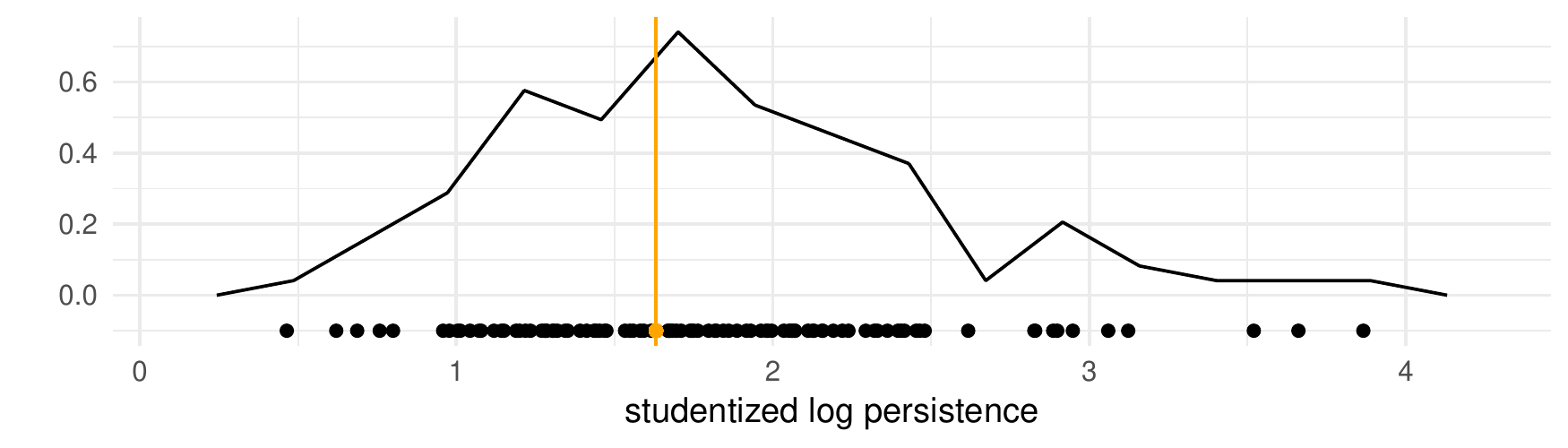} 

\end{knitrout}
\caption{Top: Mapper analysis of the Iris dataset using Petal Length as a filter function.\newline
Bottom: Density of Z-scores of log persistences from applying Method \ref{mth:global} to the Iris mapper graph.
Marked with a vertical line and a separate point, both in orange, is the maximum Z-score from the dataset itself.}
\label{fig:iris}
\end{figure}

\section{Discussion}
\label{sec:discussion}

\subsection{Certified Mapper}
\label{sec:certified-mapper}

Mapper comes close, but not quite all the way, to the nerve lemmata that pervade algebraic topology in general and persistent topology in particular.
As proposed and used, the Mapper algorithm comes with no guarantees beyond sheer luck and stability under modifying parameters for fidelity between data shape and Mapper complex shape.
It is easy to see that hidden topological structure can both introduce and hide homological features in the resulting Mapper complex, and if the structure aligns orthogonally to the Mapper filter functions, there is no way to adjust parameters to find the hidden structure.

We show an example of this in Figure \ref{fig:tubes}: an intersecting pair of cylinders in $\mathbb{R}^4$, with filter functions taken as projections onto the first variable.
Here, the structure of the two first variables -- the figure X -- is clearly seen in the Mapper graph, whereas the tube shapes -- the circles in the $z$-$w$-plane -- are completely invisible in the Mapper complex.
By the statistical multiple testing methods we describe, we get a clear indication of the resulting obstruction: at the bottom of Figure \ref{fig:tubes} we clearly
see the studentized log persistence of the intersecting cylinders to be a far outlier as compared to the null model, and in Figure \ref{fig:obstruction} we can see the shape most clearly illustrating an obstruction -- the hidden $z$-$w$ circle emerges well recognizable.

From real data, using the well studied Iris dataset, we can see an example of a lack of obstructions -- a case where we would issue a certificate and trust the fidelity of the Mapper complex to the topological features of the dataset.
Here, the studentized log persistence is close to the middle of the distribution of studentized log persistences from the null model, giving no reason to believe any Mapper cover element or cover element intersection to contain significant hidden topological features.

Mapper has found widespread use in industry, sometimes dealing with high stakes data analysis tasks.
While in practice Mapper usage often is measured on the value of identified patterns, without fidelity of shape being taken as relevant to the analysis, having a certified lack of obstructions to nerve lemmata would allow us to claim the Mapper complex shape to be a reliable descriptor of the dataset itself.
If reliability of the Mapper analysis is critical to an application, the computational cost of verifying a lack of obstructions can be a good tradeoff for higher reliance on the results.

\subsection{Multiple testing paradigms}
\label{sec:multiple-testing-paradigms}

The first idea we wish to adress is the $T(1)$-distribution of the quantile fraction introduced for Method \ref{mth:quantile-T}.
As can be seen in Figure \ref{fig:normal-quantile-distr}, the fit is not particularly convincing -- certainly not for a direct fit to $T(1)$ -- in which case the line should be a diagonal -- but even after allowing for a rescaling of the test statistic, the fit is not particularly good.

Next we would like to discuss the various tests we proposed.
We had parametric tests -- against a normal or against a $T(1)$ distribution -- for assuming a normal distribution of the $\max(d-b)$ statistics; for assuming a normal distribution of the $\log\max(d-b)$ statistics, or for assuming the $T(1)$ distribution for the observed quantiles.

Using the normal distribution directly on the $\max(d-b)$ statistic performs quite poorly: as can be seen in Table \ref{tbl:p-power}, this test rejects far too much for the null model: false positives abound.
The $p$-values for the null model come out to the range from 0.25 to 0.55 -- increasing the actual level by a factor of between 5.5 and 25.

Next, we consider the approach using a much more plausible normal approximation for $\log\max(d-b)$, this is labeled \emph{Log normal test} in Table \ref{tbl:p-power}.
We see that the level is far more reasonable here: the true level differs from the one suggested by the normal distribution by a factor of between 
1.8 and
4.
As we look to the power of this test, it performs reasonably well too -- rejection rates around 0.80 for the low noise case, and in the range between 0.37 and 0.55 for the higher noise case.

The quantile test starts out promising: the levels are even lower than the cutoffs chosen -- the test looks too conservative as long as we are looking to the null model.
However, with the heavy tails of the $T(1)$ distribution, we can notice when we try to measure power that this test simply does not reject at all.
The distribution we are comparing the test statistic with is so tail heavy that no values seem particularly extreme.

For Method \ref{mth:global}, we see dramatically different results depending on which standardization scheme we choose:

The Z-score standardization performs well: all null rejection rates are somewhat elevated from the chosen levels, and the power to recognize a signal when present is good with low noise and still present at all with higher noise.

In Figure \ref{fig:noisy-circle}, we show examples of the point clouds that the methods need to deal with -- the noise level of $\sigma=0.25$ is quite large.
The 100 point and 500 point circles look quite true, but the 10 point and 50 point circles at $\sigma=0.25$ are noisy enough that it might not be a clear call whether or not to consider the signal to be present at all.
Based on this, one may consider powers around 0.5 to be quite decent in the high noise case we use to measure power, and with both the Z-score based normalizations, Method \ref{mth:global} shows up with decent levels and good powers. 

The histogram equalization works atrociously however: just like with the quantile T-test method, this method pretty much refuses to reject the null hypothesis no matter what point clouds it sees.
On further consideration, the reason why can be seen: with histogram equalization, we are reducing the sizes of bar lengths back to a ranking, so the same issues raised against Method \ref{mth:generic} remain problematic for anything that works with histogram equalization.

\subsection{Recommendations}
\label{sec:recommendations}

We recommend using Certified Mapper whenever fidelity of shape is important to the Mapper analysis.

The persistent nerve lemma and Corollary \ref{cor:epsilon-acyclic} should be used whenever applicable to issue a quantified certificate of non-obstruction.

Where Corollary \ref{cor:epsilon-acyclic} is not applicable, we recommend the certificate of non-obstruction to be issued through a statistical method.

The uniformly distributed points in a bounding box seems to be a reasonable null model.
From \parencite{bobrowski2017maximally} we know uniform distributions to have appropriately small persistence, and the bounding box has an easily accessible unbiased estimator we can use.

On the face of it, the Z-score based global methods and the normal approximation for $\log\max(d-b)$ seem to behave equally well -- and any one of the three would be a reasonable choice.
The parametric $\log\max(d-b)$ test has better power, while the global tests have better levels.
The Z-score global method has particularly impressive power for the low noise case.

From all this, the recommendation we can see is to use the difference invariant and Method \ref{mth:global} with the Z-score normalization.
Doing this, noisier circles will be more difficult to detect, but if the signal is clean, the power of the test stays high.

\section{Future Directions}
\label{sec:future-directions}

In later work we plan to explore strategies to refine a Mapper cover to resolve any obstructions found and produce a certified Mapper complex.

\section{Acknowledgements}
\label{sec:acknowledgements}

The authors would like to acknowledge and thank: 
Sayan Mukherjee for invaluable advice and help designing Method \ref{mth:global};
Anthea Monod and Kate Turner for helpful conversations;
Dana Sylvan for giving feedback and advice on the manuscript;
The MAA for a travel grant;
The Abel Symposium for a participation and travel grant.

\printbibliography
\end{document}